\documentclass[11pt,twoside]{article}
\usepackage{fix-cm}
\usepackage{amsmath,amssymb,amsfonts,amsthm,mathrsfs}
\usepackage{graphicx,float,needspace,accents,geometry}
\usepackage[pagewise]{lineno}
\usepackage{fancyhdr,caption,titlesec,setspace,indentfirst}
\usepackage{abstract}

\usepackage[numbers,sort]{natbib}
\usepackage{enumitem}
\setlist[enumerate]{label=\textnormal{(\arabic*)},leftmargin=3em,
  labelsep=0.5em,itemsep=0pt,partopsep=0pt,parsep=0pt,topsep=3pt}
\setitemize[1]{itemsep=-4pt,partopsep=0pt,parsep=\parskip,topsep=5pt}
\setdescription{itemsep=-4pt,partopsep=0pt,parsep=\parskip,topsep=5pt}
\usepackage{microtype}
\usepackage[hyperfootnotes=false,colorlinks=true,linkcolor=blue,citecolor=blue,urlcolor=blue]{hyperref}
\allowdisplaybreaks
\makeatletter
\renewcommand{\section}{\@startsection{section}{1}{0mm}
{-\baselineskip}{0.5\baselineskip}{\normalfont\normalsize\bfseries}}
\renewcommand{\subsection}{\@startsection{subsection}{2}{0mm}
{-\baselineskip}{0.5\baselineskip}{\normalfont\normalsize\bfseries}}
\makeatother
\newcommand{\footnoteONLYtext}[1]{\begingroup
\renewcommand{\thefootnote}{}\footnotetext{#1}\endgroup}
\newcommand{\keywords}[1]{{\it Keywords}: #1}
\newcommand{\msc}[1]{{\it 2020 Mathematics Subject Classification}: #1}
\numberwithin{equation}{section}
\theoremstyle{plain}
\newtheorem{theoalph}{Theorem}

\newtheorem{prop}{Proposition}[section]
\newtheorem{lemm}[prop]{Lemma}
\newtheorem{coro}[prop]{Corollary}
\newtheorem{fact}[prop]{Fact}
\theoremstyle{definition}
\newtheorem{defi}[prop]{Definition}
\theoremstyle{remark}

\newcommand{\CC}{\widehat{\mathbb C}}

\newcommand{\sA}{\mathscr A}
\newcommand{\cG}{\mathcal G}
\newcommand{\cW}{\mathcal W}
\newcommand{\Crit}{\operatorname{Crit}}
\newcommand{\CJ}{\Crit'(f)}
\newcommand{\Jcon}{J_{\mathrm{con}}}
\newcommand{\HD}{\operatorname{HD}}
\newcommand{\Mod}{\operatorname{mod}}
\newcommand{\diam}{\operatorname{diam}}
\newcommand{\dist}{\operatorname{dist}}
\newcommand{\tB}{\widetilde B}
\newcommand{\hV}{\widehat V}
\newcommand{\Ecrit}{E_{\mathrm{crit}}(f)}

\title{\Large\bfseries Absolute winning properties of exceptional sets for rational maps}
\author{Haoyang Ji}
\date{}
\begin{document}
\maketitle
\vspace{-1cm}
\begin{abstract}
\noindent\textbf{Abstract.} We study Schmidt's game on the Julia set $J(f)$ of weak
hyperbolic rational maps $f:\CC\to\CC$. For each $y\in J(f)$,
let $E(f,y)$ be the set of points in $J(f)$ whose forward
orbits stay a positive distance from $y$. We prove that
$E(f,y)$ is absolute winning on $J(f)$ for each $y \in J(f)$. To do this, we first prove the absolute winning property for
the exceptional set associated to the critical points by a
potential game argument, and then use uniform expansion to
obtain the result for each target point $y$.
\end{abstract}
\footnoteONLYtext{Date: \today}
\footnoteONLYtext{\msc{37F10, 11J83}}
\footnoteONLYtext{\keywords{Julia set, Schmidt's game, exceptional set, absolute winning, backward contraction}}
\footnoteONLYtext{H. Ji was supported by NSFC Grant No.12301103.}

\section{Introduction}

Let $f:\CC\to\CC$ be a rational map of degree at least two,
and let $J(f)$ be its Julia set. Write
\[
 \mathcal O_f^+(x)=\{f^n(x):n\ge0\}.
\]
For $y\in J(f)$, define
\begin{equation}\label{eq:avoid}
 E(f,y)=\{x\in J(f):y\notin\overline{\mathcal O_f^+(x)}\}
 =\{x\in J(f):\inf_{n\ge0}d(f^n(x),y)>0\}.
\end{equation}
For $A\subset J(f)$, define
\begin{equation}\label{eq:exceptional-set}
 E(f,A)=\{x\in J(f):\overline{\mathcal O_f^+(x)}\cap A=\varnothing\}
 =\bigcap_{y\in A}E(f,y).
\end{equation}
Following~\cite{CG}, we call $E(f,A)$ the
$A$-exceptional set.
All closures and distances are taken in the spherical metric.

We mainly consider rational maps $f$ satisfying the weak
hyperbolicity conditions studied in~\cite{RLS}, namely
{\it expansion away from critical points} and {\it backward contraction}.
Under these assumptions, for each $y\in J(f)$, the exceptional
set $E(f,y)$ has zero spherical area measure. We shall prove that $E(f,y)$ is absolute winning on $J(f)$
for each $y\in J(f)$. Consequently, these sets are winning
for Schmidt's game and dense in $J(f)$, and their countable
intersections remain absolute winning. Under our hypotheses,
they also have full Hausdorff dimension in $J(f)$,
see~\cite[Corollary~1.3]{CG} and \cite{RLS}.

There are many results on the winning properties of exceptional
sets for various games and dynamical systems. Tseng~\cite{Tseng} established the Schmidt winning property
for $C^2$ expanding circle maps, and F\"arm~\cite{Farm} obtained
analogous results for a class of expanding Markov interval maps,
including the Gauss map. For $C^{1+\theta}$ conformal expanding
endomorphisms of compact Riemannian manifolds, Wu~\cite{Wu}
obtained the stronger property of absolute winning. For a class
of nonuniformly expanding interval maps including
Manneville--Pomeau maps, Duvall~\cite{Duvall} proved that
$E(f,0)$ is strong winning, where $0$ is the neutral fixed point.
In a recent preprint, Duvall~\cite{Duvall26} proved that
$E(f,Y)$ is absolute winning for every countable set
$Y\subset[0,1]$ for a class of Manneville--Pomeau type interval maps.
Hu-Li-Yu~\cite{HLY} proved that $E(f,y)$ is
$1/2$-strong winning for piecewise $C^{1+\delta}$ expanding
interval maps with finitely many branches.
For the relations between Schmidt winning, absolute winning,
potential winning, and Cantor winning, see
Badziahin et al.~\cite{BHNS}. For expanding rational maps, Abercrombie and Nair~\cite{AN}
proved that $E(f,y)$ has full Hausdorff dimension in $J(f)$.
For general rational maps, Campos and Gelfert~\cite{CG}
obtained lower bounds in terms of the hyperbolic dimension,
as well as full topological entropy results.

Our proof is motivated by the potential game approach
of Wu~\cite{Wu}. However, our system is not uniformly expanding because
of the presence of critical points in $J(f)$.
Thus the uniform expansion estimates used in~\cite{Wu}
are no longer available. 
To overcome this, we use the {\it polynomial shrinking} estimates
and {\it nice couples} construction in~\cite{RLS} to prove that the exceptional
set associated to the critical points in $J(f)$ is absolute winning. We then use uniform expansion away from
the critical points to obtain the result for every
$y\in J(f)$.

\subsection{Notation}

For $x,y\in\CC$ and nonempty subsets $A,B\subset\CC$, we denote
the distance between $x$ and $y$ by $d(x,y)$, the distance
between $A$ and $B$ by $\dist(A,B)$, and the diameter of $A$
by $\diam(A)$. We write $|Df(x)|$ for the norm of the
differential of $f$ at $x$. All these quantities are taken
with respect to the spherical metric.

Let $\sA_{\mathbb C}$ denote the class of rational maps
$f:\CC\to\CC$ of degree at least two without neutral periodic
points. For $f\in\sA_{\mathbb C}$, let $\Crit(f)$ denote the set
of critical points of $f$, that is, points at which $Df$ vanishes. For each $c \in \Crit(f)$, let $\ell_c\ge2$ denote the local degree of $f$ at $c$. Let $J(f)$ be its Julia set, which is a forward invariant compact set. Set
\[
 \Crit'(f)=\Crit(f)\cap J(f), \ 
 \ell_{\max}(f)=\max_{c\in\CJ}\ell_c.
\]
Whenever the latter
quantity is used, $\CJ$ is assumed nonempty.

We write $B(x,r)$ for an open
spherical ball and $\overline B_J(x,r)=J(f)\cap\overline{B(x,r)}$ for a closed
ball in the game space.

\begin{defi}\label{def:away}
We say that $f \in\sA_{\mathbb C}$ is {\it expanding away from critical points} if,
for every open neighbourhood $V$ of $\CJ$, the map $f$ is uniformly
expanding on
\[
 K(V)=\{x\in J(f):f^n(x)\notin V\text{ for every }n\ge0\}.
\]
Here uniform expansion on a compact forward invariant set $K$ means
that there are $C>0$ and $\lambda>1$ such that
$|Df^n(x)|\ge C\lambda^n$ for all $x\in K$ and $n\ge0$.
\end{defi}

We now recall the backward contraction property, introduced
in~\cite{RL} for rational maps and in~\cite{BRLSvS08} for
interval maps. A {\it pull-back} of a set $V$ by $f^n$ is a connected component of
$f^{-n}(V)$. For $c\in\Crit(f)$ and $\delta>0$, let
$\tB(c,\delta)$ be the component of
$f^{-1}(B(f(c),\delta))$ containing $c$. We also put
\[
 \tB(\CJ,\delta)=\bigcup_{c\in\CJ}\tB(c,\delta).
\]
There are constants $a_0,a_1>0$, depending on $f$, such that, for
every critical point $c$ and all sufficiently small $\delta>0$,
\begin{equation}\label{eq:local-degree}
 B(c,a_0\delta^{1/\ell_c})\subset\tB(c,\delta)
 \subset B(c,a_1\delta^{1/\ell_c}).
\end{equation}

\begin{defi}\label{def:bc}
Let $r>1$. We say that $f \in\sA_{\mathbb C}$ is {\it backward contracting with
exponent $r$}, abbreviated
$BC(r)$, if there is $\delta_0=\delta_0(r)>0$ such that, for every
$\delta\in(0,\delta_0)$, every $c\in\CJ$, every $n\ge0$, and
every pull-back $W$ of $\tB(c,r\delta)$ by $f^n$,
\[
 \dist(W,f(\Crit(f)))\le\delta
  \ \Longrightarrow \  \diam(W)<\delta.
\]
The map $f$ is called {\it backward contracting}
if it satisfies $BC(r)$ for every $r>1$.
\end{defi}

A theorem of Ma\~n\'e asserts that every $C^3$ multimodal interval map with non-flat critical points is expanding away from critical points;
see~\cite{Man85}. Although the analogous statement for a map
$f\in\sA_{\mathbb C}$ is false in general, it does hold if we
assume in addition that $f$ is a polynomial that is at most
finitely renormalizable; see~\cite{KvS09}, and also~\cite{QY09}
for the totally disconnected case. Moreover, the following fact gives sufficient conditions for expansion
away from critical points and backward contraction.

\begin{fact}\label{fact:backward-contracting}
A map $f\in\sA_{\mathbb C}$ is expanding away from critical points and backward contracting if one of the following holds:
\begin{enumerate}

\item
$f$ is a polynomial that is at most
finitely renormalizable and, for every $c\in\Crit'(f)$,
\[
 |Df^n(f(c))|\longrightarrow\infty
  \ \text{as }n\to\infty.
\]

\item
$f$ is a rational map such that,
for every $c\in\Crit'(f)$,
\[
 \sum_{n=0}^{\infty}|Df^n(f(c))|^{-1}<\infty.
\]
\end{enumerate}
\end{fact}

\begin{proof}
The backward contraction property follows from
\cite[Theorem~A]{LS10} in the first case and from
\cite[Theorem~A]{RL} in the second case.
Expansion away from critical points follows from
\cite{KvS09} or \cite[Corollary~8.3]{RL} in the first case,
and from \cite[proof of Lemma~3.1]{Prz98} in the second case.
\end{proof}

\subsection{Statement of results}\label{sec:statements}

Our main result is the following.

\begin{theoalph}\label{thm:rls}
Let $f\in\sA_{\mathbb C}$, with $\CJ\ne\varnothing$, be expanding
away from critical points and backward contracting. Then
$E(f,y)$ is absolute winning on $J(f)$ for every $y\in J(f)$.
\end{theoalph}

Theorem~\ref{thm:rls} and
\eqref{eq:exceptional-set} imply that $E(f,Y)$ is absolute
winning on $J(f)$ for every countable set $Y\subset J(f)$.

\begin{coro}\label{cor:strong-winning}
Under the hypotheses of Theorem~\ref{thm:rls}, there exists
$\alpha\in(0,1)$, depending only on $J(f)$, such that
$E(f,Y)$ is $\alpha$-strong winning on $J(f)$ for every
countable set $Y\subset J(f)$.
\end{coro}

\begin{proof}
Rational Julia sets are uniformly perfect~\cite{Eremenko,ManRocha92}.
By Theorem~\ref{thm:rls} and the countable intersection property,
$E(f,Y)$ is absolute winning on $J(f)$. The corollary therefore
follows from Proposition~\ref{prop:absolute-schmidt} with
$\alpha=\kappa/8$, where $\kappa$ is as in~\eqref{eq:uniform-perfect}
for $X=J(f)$.
\end{proof}

\begin{coro}\label{cor:nondense}
Under the hypotheses of Theorem~\ref{thm:rls}, the set
\[
 \{x\in J(f):\{f^n(x):n\ge0\}\text{ is not dense in }J(f)\}
\]
is absolute winning on $J(f)$.
\end{coro}
\begin{proof}
It contains $E(f,y)$ for any fixed $y\in J(f)$, which is
absolute winning on $J(f)$ by Theorem~\ref{thm:rls}.
\end{proof}

For $\CJ\ne\varnothing$, put
\begin{equation}\label{eq:Ecrit}
 \Ecrit=\left\{x\in J(f):
 \inf_{n\ge0}\dist(f^n(x),\CJ)>0\right\}
 =E(f,\CJ).
\end{equation}
To prove Theorem~\ref{thm:rls}, we first establish the absolute winning property 
of $\Ecrit$ and then use expansion away from critical points to pass to arbitrary target points.

\begin{theoalph}\label{thm:reduction}
Let $f\in\sA_{\mathbb C}$ with $\CJ\ne\varnothing$.
Assume that $f$ is expanding away from critical points
and that $\Ecrit$ is absolute winning on $J(f)$.
Then $E(f,y)$ is absolute winning on $J(f)$ for every $y\in J(f)$.
\end{theoalph}

This article is organized as follows. In 
Section~\ref{sec:games} we recall background on Schmidt's game. In Section~\ref{sec:expanding}, we prove
Theorem~\ref{thm:reduction}. In Section~\ref{sec:critical},
we show that $\Ecrit$ is absolute winning on $J(f)$
and deduce Theorem~\ref{thm:rls}. In Section~\ref{sec:interval} we prove a similar result for smooth interval maps.

\section{Preliminaries}\label{sec:games}
Throughout this section, $(X,d)$ is a nonempty complete metric
space. For $x\in X$ and $r>0$, let
\[
 \overline B_X(x,r)=\{y\in X:d(x,y)\le r\}
\]
denote the closed ball centred at $x$ with radius $r$. 
Following the convention discussed in~\cite[Remark~2.1]{BHNS},
we regard each ball as specified by its centre and radius. We write
\[
 \overline B_X(x',r')\preceq\overline B_X(x,r)
 \ \text{if} \ d(x',x)+r'\le r.
\]
This condition implies
$\overline B_X(x',r')\subseteq\overline B_X(x,r)$.
The nesting conditions in the games below are imposed
using $\preceq$. For rational maps, the game space is $X=J(f)$
with the spherical metric.

\subsection{Schmidt's game}
Let $\alpha,\beta\in(0,1)$ be two constants and $S\subset X$ be a target set.  In Schmidt's $(\alpha,\beta)$-game,
Bob first chooses $B_0=\overline B_X(b_0,r_0)$. After $B_i$ has been chosen,
Alice chooses $A_i=\overline B_X(a_i,\alpha r_i)\preceq B_i$, and Bob then
chooses
\[
 B_{i+1}=\overline B_X(b_{i+1},\alpha\beta r_i)\preceq A_i.
\]
Thus $r_i=(\alpha\beta)^ir_0 \to 0$, and completeness gives a single point $x_\infty\in\bigcap_{i\ge0}B_i$. Alice wins precisely when $x_\infty\in S$.

\begin{defi}\label{def:schmidt}
A set $S\subset X$ is {\it $(\alpha,\beta)$-winning} if Alice has a strategy for placing her balls that ensures winning no matter how Bob chooses his balls. It is called {\it $\alpha$-winning} if it is $(\alpha,\beta)$-winning for
every $\beta\in(0,1)$. It is called {\it winning} if it is $\alpha$-winning for
some $\alpha \in(0,1)$
\end{defi}

\begin{prop}[\cite{Schmidt,McMullen,BFKRW}]\label{prop:schmidt-properties}
The $\alpha$-winning sets have the following properties.
\begin{enumerate}
\item Every $\alpha$-winning set is dense in $X$.
\item If $S_i\subset X$ is $\alpha$-winning for every $i\ge1$,
then $\bigcap_{i\ge1}S_i$ is $\alpha$-winning.
\item If $S\subset X$ is $\alpha$-winning, then it is
$\alpha^*$-winning for every $0<\alpha^*<\alpha$.
\item If $S\subset X$ is $\alpha$-winning and $S\ne X$,
then $\alpha\le1/2$.
\item If $X=\mathbb R^d$ and $S\subset X$ is $\alpha$-winning,
then $\HD(S)=d$. More precisely,
\begin{equation}\label{eq:schmidt-dimension}
 \HD(S\cap U)=d
 \ \text{for every nonempty open }U\subset\mathbb R^d.
\end{equation}
\end{enumerate}
\end{prop}

\subsection{The strong game}
McMullen's strong game~\cite{McMullen}, see also \cite{FPS10}, has the same nesting
rules as Schmidt's game, but the radius equalities are replaced
by inequalities. Fix $\alpha,\beta\in(0,1)$. After Bob chooses
$B_i=\overline B_X(b_i,r_i)$, Alice chooses
$A_i=\overline B_X(a_i,s_i)\preceq B_i$ with $s_i\ge\alpha r_i$,
and Bob chooses $B_{i+1}\preceq A_i$ with
$r_{i+1}\ge\beta s_i$. Alice wins for $S\subset X$ if
$S\cap\bigcap_{i\ge0}B_i\ne\varnothing$.

\begin{defi}\label{def:strong}
A set $S\subset X$ is {\it $(\alpha,\beta)$-strong winning}
if Alice has a winning strategy for placing her balls that ensures
winning no matter how Bob chooses his balls. It is called
{\it $\alpha$-strong winning} if it is $(\alpha,\beta)$-strong
winning for every $\beta\in(0,1)$. It is called {\it strong winning}
if it is $\alpha$-strong winning for some $\alpha\in(0,1)$.
\end{defi}

It is easy to see that strong winning sets are winning. Proposition~\ref{prop:schmidt-properties} remains valid with
`winning' replaced by `strong winning'. Note that the strong winning property is preserved under quasisymmetric homeomorphisms, which is not generally true of the winning property, see \cite{McMullen}.

\subsection{The absolute game}

The absolute game was introduced by McMullen in~\cite{McMullen}. It is played similarly to Schmidt's game but instead of choosing a region where Bob must play, Alice now chooses a region where Bob must not play. 
Fix $\beta\in(0,1)$. After Bob chooses $B_i=\overline B_X(b_i,r_i)$, Alice
chooses $A_i=\overline B_X(a_i,\rho_i)$ with $0<\rho_i\le\beta r_i$.
Bob must respond with a ball satisfying
\begin{equation}\label{eq:absolute-rules}
 B_{i+1}\preceq B_i, \  B_{i+1}\cap A_i=\varnothing,
  \  r_{i+1}\ge\beta r_i.
\end{equation}
There is no requirement that $A_i\subset B_i$. Alice wins by default if Bob cannot choose a next ball
satisfying~\eqref{eq:absolute-rules}, or if $r_i$ does not tend
to zero. Otherwise she wins for $S$ if
$\bigcap_{i\ge0}B_i=\{x_\infty\}\subset S$.

\begin{defi}\label{def:absolute}
A set $S\subset X$ is {\it $\beta$-absolute winning} if Alice
has a winning strategy with parameter $\beta$. It is
{\it absolute winning} if it is $\beta$-absolute winning for
every $\beta\in(0,1)$.
\end{defi}

We follow the definition of the absolute game on metric spaces
in~\cite[Section~2.2]{BHNS}. We do not require $0<\beta<1/3$ as in the original definitions in~\cite{McMullen}. It is straightforward to see that for $0<\beta<\beta'<1$, every $\beta$-absolute winning set
is also $\beta'$-absolute winning. Indeed, increasing $\beta$ enlarges Alice's
choices and restricts Bob's choices. On Euclidean space one
usually restricts the definition to $0<\beta<1/3$ to ensure that
Bob can respond.

\begin{defi}
For $\kappa\in(0,1)$, we say that $X$ is
{\it $\kappa$-uniformly perfect} if there exists $r_*>0$ such that, for each
$x\in X$ and $0<r\le r_*$, there is $w\in X$ with
\begin{equation}\label{eq:uniform-perfect}
 \kappa r\le d(x,w)\le r.
\end{equation}
If $X$ is $\kappa$-uniformly perfect for some $\kappa\in(0,1)$,
then we say that it is {\it uniformly perfect}.
\end{defi}

Condition~\eqref{eq:uniform-perfect} ensures that Bob has a legal
choice whenever $r_i\le r_*$ and $0<\beta\le\kappa/8$.

\begin{prop}\label{prop:absolute-schmidt}
If $X$ satisfies \eqref{eq:uniform-perfect}, then every absolute
winning subset of $X$ is both $\alpha$-strong winning and
$\alpha$-winning, where $\alpha=\kappa/8$.
\end{prop}

\begin{proof}
Let $S\subset X$ be absolute winning. It suffices to show that $S$ is $\alpha$-strong winning.
Fix $\beta\in(0,1)$ for the strong game and put
$\gamma=\alpha\beta$, and fix a $\gamma$-absolute winning strategy $\sigma$ for $S$.
Alice plays the strong game as follows: she always chooses a ball of radius
exactly $\alpha r_i$. Since $B_{i+1}\preceq A_i$, we have $r_{i+1}\le\alpha r_i$, and hence $r_i \to 0$.
After finitely many initial plays and relabelling, we may
assume that $r_i\le r_*$ for all $i\ge0$.

Given Bob's ball $B_i$ in the strong game, Alice applies the strategy $\sigma$ to $B_i$
and obtains the ball $D_i$ that $\sigma$ deletes. She then chooses a ball $A_i$ of radius
$\alpha r_i$ such that $A_i\preceq B_i$ and
$A_i\cap D_i=\varnothing$.
Bob's next ball must then lie in $A_i$ and hence avoid $D_i$. Condition~\eqref{eq:uniform-perfect} guarantees that such a ball
$A_i$ exists; we construct it as follows.

Write $B_i=\overline B_X(x,r)$ and
$D_i=\overline B_X(a,\rho)$, where $\rho\le\gamma r$. Apply
\eqref{eq:uniform-perfect} at radius $r/2$ to obtain $w$ with
$\kappa r/2\le d(x,w)\le r/2$. One of $x,w$, denoted by $z$,
satisfies $d(z,a)\ge\kappa r/4$. Thus
$A_i:=\overline B_X(z,\alpha r)\preceq B_i$,
and $A_i$ is disjoint from $D_i$, since
\[
 (\alpha+\gamma)r=\kappa(1+\beta)r/8<\kappa r/4\le d(z,a).
\]
Bob's next ball avoids $D_i$, is formally nested in $B_i$,
and has radius at least $\alpha\beta r=\gamma r$. Thus Bob's next ball also satisfies the rules of the
$\gamma$-absolute game. Since $r_i \to 0$ by the choice of Alice,
$\bigcap_{i\ge0}B_i=\{x_\infty\}\subset S$ by the absolute strategy.
\end{proof}

\begin{prop}[\cite{McMullen,BFKRW,FSU}]\label{prop:absolute-properties}
Absolute winning sets have the following properties.
\begin{enumerate}
\item If $X$ is uniformly perfect, then every absolute winning
subset of $X$ is dense in $X$.
\item If $S_i\subset X$ is absolute winning for every $i\ge1$,
then $\bigcap_{i\ge1}S_i$ is absolute winning on $X$.
\end{enumerate}
\end{prop}

The absolute winning property on $\mathbb R^d$ is also preserved
under quasisymmetric homeomorphisms; see~\cite[Theorem~2.2]{McMullen}. On $\mathbb R^d$, absolute winning implies $\alpha$-strong winning for
every $0<\alpha<1/2$~\cite{McMullen}.

\subsection{The potential game}

Fix $s>0$ and $\beta\in(0,1)$. In the $(s,\beta)$-potential game, Bob first chooses an initial ball
$B_0=\overline B_X(b_0,r_0)$. After $B_i$ is chosen, Alice specifies an
at most countable collection of balls
$A_{i,j}=\overline B_X(a_{i,j},\rho_{i,j})$, of positive radii, such that
\begin{equation}\label{eq:budget}
 \sum_j\rho_{i,j}^{s}\le(\beta r_i)^s.
\end{equation}
The collection may be empty. Bob then chooses
\begin{equation}\label{eq:potential-bob}
 B_{i+1}\preceq B_i, \  r_{i+1}\ge\beta r_i.
\end{equation}
He need not avoid any of Alice's balls. If $r_i\to0$, Alice wins
for $S$ precisely when
\begin{equation}\label{eq:potential-outcome}
 x_\infty\in S\cup\bigcup_{i,j}A_{i,j}.
\end{equation}
She wins by default if the radii do not tend to zero.
This is the point version of the potential game introduced 
in~\cite[Appendix~C]{FSU}.

\begin{defi}\label{def:potential}
The set $S$ is {\it $(s,\beta)$-potential winning} if Alice has
a winning strategy. For $c_0\ge0$, it is
{\it $c_0$-potential winning} if it is $(s,\beta)$-potential
winning for every $s>c_0$ and every $\beta\in(0,1)$.
\end{defi}

For $0<s<s'$ and $\beta\in(0,1)$, every $(s,\beta)$-potential
winning set is $(s',\beta)$-potential winning, since increasing
$s$ enlarges Alice's choices.

\begin{prop}\label{prop:potential-properties}
Let $X$ be complete and $c_0\ge0$.
\begin{enumerate}
\item If $S_j\subset X$ is $c_0$-potential winning for every
$j\ge1$, then $\bigcap_{j\ge1}S_j$ is $c_0$-potential winning
on $X$.
\item If $S$ is $c_0$-potential winning and $K\subset X$ is
closed, then $S\cap K$ is $c_0$-potential winning on $K$.
\item If $\HD(N)\le c_0$, then $X\setminus N$ is
$c_0$-potential winning. Consequently, if $S$ is $c_0$-potential winning and
$\HD(N)\le c_0$, then $S\setminus N$ is
$c_0$-potential winning.
\end{enumerate}
\end{prop}

\begin{proof}
For (1), we extend the argument of~\cite[Proposition~2.4]{BHNS}
to countable intersections. Fix $s>c_0$ and $\beta\in(0,1)$,
and let $\beta_j=\beta 2^{-j/s}$ for $j\ge1$.
Bob first chooses a ball $B_0$. Since each $S_j$ is
$c_0$-potential winning, Alice has an $(s,\beta_j)$-winning
strategy $\sigma_j$ for $S_j$; fix one such strategy for each $j$.
In response to $B_0$, the strategy $\sigma_j$ gives an
at most countable collection of balls $\mathcal A_{0,j}$.
Alice specifies all the balls in $\bigcup_{j\ge1}\mathcal A_{0,j}$
as her first move. This collection is still at most countable.

Bob then chooses a single ball $B_1\preceq B_0$ with
$r_1\ge\beta r_0$; he need not avoid any of Alice's balls.
This choice is valid for every $(s,\beta_j)$-game,
since $\beta_j\le\beta$. Each strategy $\sigma_j$ now gives
its next choice $\mathcal A_{1,j}$, and Alice again specifies
the union of these collections. She continues in this way at
every turn. For each $i$, we have
\[
 \sum_{j\ge1}\sum_{A\in\mathcal A_{i,j}}
 \operatorname{rad}(A)^s
 \le\sum_{j\ge1}(\beta_jr_i)^s=(\beta r_i)^s,
\]
where $\operatorname{rad}(A)$ denotes the radius of $A$.
Thus all of Alice's choices satisfy~\eqref{eq:budget}.
If $r_i\to0$ and the outcome lies outside all the prescribed
balls, each strategy $\sigma_j$ guarantees that it belongs to
$S_j$, so it belongs to $\bigcap_{j\ge1}S_j$.
In all other cases Alice wins by default.

Statement (2) follows from~\cite[Lemma~4.4]{BHNS}; its proof
also applies when $c_0=0$, since the argument only uses $s>c_0$.
Statement (3) is~\cite[Proposition~4.6]{BHNS}.
\end{proof}

A metric space $(X,d)$ is called {\it doubling} if there exists
an integer $M\ge1$ such that, for every $x\in X$ and $r>0$,
there are points $x_1,\ldots,x_M\in X$ satisfying
\[
 \overline B_X(x,r)\subseteq
 \bigcup_{j=1}^{M}\overline B_X(x_j,r/2).
\]
The following proposition is due to Fishman, Simmons, and
Urba\'nski~\cite[Appendix~C, Observation~C.7 and Theorem~C.8]{FSU};
see also~\cite[Proposition~1.3]{BHNS}.

\begin{prop}\label{prop:conversion}
On a complete doubling metric space,
\[
 S\text{ is absolute winning}
\ \Longleftrightarrow \
 S\text{ is }0\text{-potential winning}.
\]
\end{prop}

Consequently, on a complete doubling uniformly perfect space,
\begin{equation}\label{eq:game-implications}
 0\text{-potential winning}
 \ \Longleftrightarrow\ \text{absolute winning}
 \ \Longrightarrow\ \text{strong winning}.
\end{equation}

For our applications $J(f)$ is compact, doubling, and uniformly
perfect~\cite{Eremenko,ManRocha92}. Thus
\eqref{eq:game-implications} applies directly.
In particular, to prove Theorem A, it suffices to prove that $E(f, y)$ is $(s,\beta)$-potential winning
for every $s>0$ and $\beta\in(0,1)$. We use this criterion in
Sections~\ref{sec:expanding} and~\ref{sec:critical}.

\section{Uniform expansion and absolute winning}
\label{sec:expanding}

Throughout this section, let $f\in\sA_{\mathbb C}$,
and equip $\CC$ with the spherical metric.
The proof below adapts the argument in~\cite[Section~4]{Wu} and is given in
Subsection~\ref{sec:expanding-proofs}.

\begin{prop}\label{prop:expanding}
Let $K\subset J(f)$ be compact and forward invariant. Suppose that
$f$ is uniformly expanding on $K$. Then, for each $y\in J(f)$,
\begin{equation}\label{eq:conditional-set}
 S(K,y)=(J(f)\setminus K)\cup E(f,y)
\end{equation}
is absolute winning on $J(f)$.
\end{prop}

If $K=\varnothing$ or $y\notin K$, then $S(K,y)=J(f)$.
Indeed, in the second case compactness and forward invariance give
$K\subset E(f,y)$. This case is trivial, so we may assume that $y\in K$.

\subsection{Inverse branches}

\begin{lemm}\label{lem:extension}
There are an open neighbourhood $O \subset \CC$ of $K$, constants $A\in(0,1]$
and $\lambda>1$, and $\delta>0$ with the following properties.
\begin{enumerate}
\item If $z,f(z),\ldots,f^{n-1}(z)\in O$, then
\begin{equation}\label{eq:segment-expansion}
 |Df^n(z)|\ge A\lambda^n.
\end{equation}
\item For each $x\in K$ and $n\ge0$ there is an inverse branch
$\psi$ of $f^n$ on $B(f^n(x),4\delta)$ such that
$\psi(f^n(x))=x$ and
\[
 f^k\bigl(\psi(B(f^n(x),4\delta))\bigr)\subset O
  \ (0\le k<n).
\]
\end{enumerate}
\end{lemm}

\begin{proof}
Choose $N\ge1$ such that $\inf_K|Df^N|>4$. Since the derivatives
of $f^r$, $0\le r\le N$, do not vanish on $K$, there are a
neighbourhood $O$ of $K$ and $m>0$ such that
\[
 |Df^N|>2, \  \min_{0\le r<N}|Df^r|>m
  \ \hbox{on }O.
\]
Write $n=qN+r$, where $0\le r<N$. Assume that $z,f(z),\ldots,f^{n-1}(z)\in O$. By the chain rule,
\[
 |Df^n(z)|
 =|Df^r(f^{qN}(z))|
   \prod_{j=0}^{q-1}|Df^N(f^{jN}(z))| \geq 2^q |Df^r(f^{qN}(z))| \geq 2^q\min\{1,m\}.
\]
Setting $\lambda=2^{1/N}$, we have
$2^q=\lambda^{n-r}\ge\lambda^n/2$, since $r<N$. Therefore
\[
 |Df^n(z)|\ge\frac{\min\{1,m\}}{2}\lambda^n,
\]
which proves~\eqref{eq:segment-expansion} with
$A=\min\{1,m\}/2$.

Now we prove statement (2). Put $g=f^N$. For each $x\in K$, there exists an open neighbourhood $U_x$ of $x$ on which $g$ is injective. Shrink $U_x$ if necessary so that $f^k(U_x)\subset O$ for all 
$0 \leq k < N$. Choose $a_x>0$ such that
$B(g(x),2a_x)\subset g(U_x)$, and let
\[
 H_x=(g|_{U_x})^{-1}\big|_{B(g(x),2a_x)}.
\]
Then $ |DH_x(u)| < 1/2$ for every $u\in B(g(x),2a_x)$. Let $W_x=H_x(B(g(x),a_x))$, which forms a cover of $K$. Hence we can choose a finite subcover $W_{x_1},\ldots,W_{x_\ell}$ and let
$a=\min_{1\le j\le\ell}a_{x_j}$, decreasing $a$ if necessary
so that spherical balls of radius $a$ are geodesically convex. For any $x\in K$, choose $j$ with $x\in W_{x_j}$. Then $ B(g(x),a)\subset B(g(x_j),2a_{x_j})$, so $h_x:=H_{x_j}|_{B(g(x),a)}$ is an inverse branch of $g$
satisfying $h_x(g(x))=x$, $|Dh_x|<1/2$, and
$f^k(h_x(B(g(x),a)))\subset O$ for $0\le k<N$. Hence $h_x(B(g(x),a))\subset B(x,a/2)$.

Applying the same argument to $f^r$, $0\le r<N$, we obtain
$0 < b \leq a$ such that, for every $x\in K$ and $0\le r<N$,
there is an inverse branch $h_{x,r}$ of $f^r$ on $B(f^r(x),b)$
satisfying $h_{x,r}(f^r(x))=x$, $|Dh_{x,r}|<1/m$, and
$f^k(h_{x,r}(B(f^r(x),b)))\subset O$ for $0\le k<r$.
Hence $h_{x,r}(B(f^r(x),b))\subset B(x,b/m)$.
Choose $4\delta\le\min\{b,am\}$, so that
$h_{x,r}(B(f^r(x),4\delta))\subset B(x,a)$.
For $n=qN+r$, we first apply the inverse branch of $f^r$
and then successively apply $q$ inverse branches of $f^N$.
This finishes the proof.

\end{proof}

In the rest we fix $y\in K$ and let $\delta>0$
be as in Lemma~\ref{lem:extension}.
Let $\cG_n=\cG_n(y,\delta)$ denote the collection of inverse
branches $\phi$ of $f^n$ on $B(y,2\delta)$ such that
\begin{equation}\label{eq:good-branches}
 f^k\bigl(\phi(B(y,2\delta))\bigr)\subset O
  \ (0\le k<n).
\end{equation}
We write $\cG_n$ for $\cG_n(y,\delta)$ below.
Set $\cG_0=\{\mathrm{id}\}$. Each $\cG_n$ is finite.
For $0\le k\le n$ and $0<u\le2\delta$,
\begin{equation}\label{eq:intermediate-diam}
 \diam\bigl(f^k\circ\phi(B(y,u))\bigr)
 \le2A^{-1}\lambda^{-(n-k)}u.
\end{equation}
Indeed, the derivative of this inverse branch has norm at most
$A^{-1}\lambda^{-(n-k)}$ by
\eqref{eq:segment-expansion}, and integrating this derivative along shortest spherical geodesics
from $y$ to points of $B(y,u)$ gives the desired estimate. In particular, $\diam(\phi(B(y,2\delta)))\le4A^{-1}\delta$.

We use the following spherical Koebe estimates. Decrease $\delta$
so that $\delta<r_{\mathrm K}$. The proposition then applies to every
$\phi\in\cG_n$ with $v=y$ and $r=\delta$.

\begin{prop}[Spherical Koebe estimates {\cite[Section~2.4]{PRL07}}]\label{prop:spherical-koebe}
Let $f$ be a rational map of degree at least two. There are constants
$r_{\mathrm K}>0$ and $C\ge1$, depending only on $f$, such that the following
holds. For every $v\in J(f)$, $0<r<r_{\mathrm K}$, $n\ge0$, and every
inverse branch $\phi$ of $f^n$ on $B(v,2r)$, we have
\begin{equation}\label{eq:spherical-koebe-distortion}
 C^{-1}\le\frac{|D\phi(z)|}{|D\phi(w)|}\le C
 \ \text{for all }z,w\in B(v,r).
\end{equation}
Moreover, for every $0<t\le r$,
\begin{equation}\label{eq:spherical-koebe-balls}
 B(\phi(v),C^{-1}t|D\phi(v)|)
 \subset\phi(B(v,t))
 \subset B(\phi(v),Ct|D\phi(v)|),
\end{equation}
and
\begin{equation}\label{eq:spherical-koebe-lip}
 \frac{d(\phi(z),\phi(w))}{d(z,w)}\le C|D\phi(v)|
 \ \text{for all distinct }z,w\in B(v,r).
\end{equation}
\end{prop}

\begin{lemm}\label{lem:captureK}
Let $0<\epsilon<\delta$. If $x\in K$ and
$f^n(x)\in B(y,\epsilon)$, there is $\phi\in\cG_n$ such that
$x\in\phi(B(y,\epsilon))$.
\end{lemm}

\begin{proof}
The ball $B(y,2\delta)$ is contained in
$B(f^n(x),4\delta)$. The restriction of the inverse branch obtained from
Lemma~\ref{lem:extension} to $B(y,2\delta)$ is the required
map $\phi\in\cG_n$.
\end{proof}

For each $0<t<\delta/4$ and $\phi\in\cG_n$, put
\[
 W_\phi=W_\phi(t)=\phi(B(y,t)), \ 
 p_\phi=\phi(y), \ D_\phi=\diam(W_\phi).
\]
Both $W_\phi$ and $D_\phi$ depend on $t$ and $\phi$;
we omit $t$ from the notation below.
Since $y \in K \subset J(f)$, $p_\phi$ also belongs to $J(f)$.
Proposition~\ref{prop:spherical-koebe} gives constants
$a>0$ and $H,Q\ge1$, independent of $t,n,\phi$, such that
\begin{align}
 B(p_\phi,a\delta D_\phi/t)
 &\subset\phi(B(y,\delta)),\label{eq:large-image}\\
 \phi(B(y,\epsilon))
 &\subset B(p_\phi,H\epsilon D_\phi/t)
  \ (\forall 0<\epsilon<t),\label{eq:small-image}\\
 \frac{d(\phi(z),\phi(w))}{d(z,w)}&\le \frac{QD_\phi}{t}
 \ \bigl(z,w\in B(y,\delta),\ z\ne w\bigr).
 \label{eq:branch-lip}
\end{align}
Indeed, \eqref{eq:spherical-koebe-balls} gives
\[
 C^{-1}t|D\phi(y)|\le D_\phi\le2Ct|D\phi(y)|.
\]
Using \eqref{eq:spherical-koebe-balls} and
\eqref{eq:spherical-koebe-lip}, we may therefore take
$a=(2C^2)^{-1}$ and $H=Q=C^2$. Also,
\begin{equation}\label{eq:diam-global}
 D_\phi\le2A^{-1}t\lambda^{-n}.
\end{equation}

\begin{lemm}\label{lem:countK}
For every $\beta\in(0,1)$ there is an integer $L_\beta\ge1$ with
the following property. For each
\begin{equation}\label{eq:t-choice}
 0<t<\delta/4 \mbox{ with } a\delta/t\ge20/\beta,
\end{equation}
for every $x\in J(f)$, $r>0$, and $0<\eta\le1$, the number
of branches $\phi\in\bigcup_{n\ge0}\cG_n$ satisfying
\begin{equation}\label{eq:active}
 r\le D_\phi<r/\beta, \ 
 \overline B_J(p_\phi,\eta D_\phi)\cap \overline B_J(x,r)\ne\varnothing
\end{equation}
is at most $L_\beta$.
\end{lemm}

\begin{proof}
It suffices
to prove the bound for $\eta=1$.
For any branch in \eqref{eq:active}, we then have
\[
 d(p_\phi,x)\le r+D_\phi<r+r/\beta\le2r/\beta,
\]
and hence $W_\phi\subset B(x,3r/\beta)$.
For any two branches $\phi,\phi'$ satisfying \eqref{eq:active},
these bounds give
\[
 W_{\phi'}\subset B(x,3r/\beta)
 \subset B(p_\phi,5r/\beta).
\]
Since $D_\phi\ge r$, condition~\eqref{eq:t-choice} gives
$a\delta D_\phi/t\ge20r/\beta$. Hence, by
\eqref{eq:large-image},
\[
 W_{\phi'}\subset B(p_\phi,5r/\beta)
 \subset B(p_\phi,a\delta D_\phi/t)
 \subset\phi(B(y,\delta)).
\]

Now we show that at each fixed time $n$ there is at most one branch satisfying \eqref{eq:active}. In fact, if
$\phi_1,\phi_2\in\cG_n$ satisfy \eqref{eq:active}, then
$W_{\phi_2}\subset\phi_1(B(y,\delta))$. The map $f^n$ is
injective on $\phi_1(B(y,\delta))$ and maps $W_{\phi_2}$ onto $B(y,t)$.
Thus $W_{\phi_2}=W_{\phi_1}$, and the inverse branches agree.

Now let $\phi_i\in\cG_{n_i}$ satisfy \eqref{eq:active}, with
$n_1\le n_2$. The same argument gives
$f^{n_1}(W_{\phi_2})\subset B(y,\delta)$. Applying
\eqref{eq:intermediate-diam} to $f^{n_1}\circ\phi_2$ and then
\eqref{eq:branch-lip} to $\phi_1$, we obtain
\[
 D_{\phi_2}
 \le 2QA^{-1}\lambda^{-(n_2-n_1)}D_{\phi_1}.
\]
Let $M=\max\{1,2QA^{-1}\}$. Since
$r\le D_{\phi_i}<r/\beta$ for $i=1,2$, we have
\[
 \beta<\frac{D_{\phi_2}}{D_{\phi_1}}
 \le M\lambda^{-(n_2-n_1)}
\ \Longrightarrow \ 
 n_2-n_1<\frac{\log(M/\beta)}{\log\lambda}.
\]
It is therefore enough to take
\[
 L_\beta=2+\left\lceil\frac{\log(M/\beta)}{\log\lambda}\right\rceil.
\qedhere
\]
\end{proof}

\subsection{Proofs of Proposition~\ref{prop:expanding} and Theorem~\ref{thm:reduction}}
\label{sec:expanding-proofs}
When playing the games on $J(f)$, we use the closed balls $\overline B_J(x,r)$
with centres $x\in J(f)$. Fix $O,A,\lambda,\delta$ as in Lemma~\ref{lem:extension},
with $\delta$ decreased so that Proposition~\ref{prop:spherical-koebe}
applies. Fix its constant $C$, set $a=(2C^2)^{-1}$ and
$H=Q=C^2$ as above. For each $\beta \in (0, 1)$, let $L_{\beta}$ be given by Lemma~\ref{lem:countK}, uniformly for all $t$ satisfying~\eqref{eq:t-choice}.

\begin{proof}[Proof of Proposition~\ref{prop:expanding}]

Fix $y\in J(f)$. The case $y\notin K$ is trivial,
so we assume $y\in K$ and use the families $\cG_n$ for this fixed $y$ and $\delta$.
By Proposition~\ref{prop:conversion}, it suffices to show that
$S(K,y)$ is $0$-potential winning on $J(f)$.

Fix $s>0$ and $\beta\in(0,1)$ arbitrarily. Choose $\eta\in(0,1)$ such that
\begin{equation}\label{eq:etaK}
 L_\beta(\eta/\beta)^s\le\beta^s.
\end{equation}
After Bob first chooses $B_0=\overline B_J(b_0,r_0)$, choose $t$ satisfying
\eqref{eq:t-choice} and $2A^{-1}t<r_0$. Then choose
$\epsilon\in(0,t)$ such that $H\epsilon/t\le\eta$.
Alice chooses the empty collection $\varnothing$ on the first turn.

At each turn $i\ge1$, after Bob chooses $B_i=\overline B_J(b_i,r_i)$,
Alice chooses the balls $\overline B_J(p_\phi,\eta D_\phi)$ corresponding
to all branches $\phi\in\bigcup_{n\ge0}\cG_n$ satisfying
\begin{equation}\label{eq:crossingK}
 r_i\le D_\phi<r_{i-1}
\end{equation}
and $\overline B_J(p_\phi,\eta D_\phi)\cap B_i\ne\varnothing$. Since
$r_i\ge\beta r_{i-1}$, Lemma~\ref{lem:countK} bounds their
number by $L_\beta$. The radius of each chosen ball satisfies
\[
 \eta D_\phi<\eta r_{i-1}\le\frac{\eta}{\beta}r_i.
\]
Since there are at most $L_\beta$ such balls, by \eqref{eq:etaK} 
\[
 \sum_{\phi}(\eta D_\phi)^s
 \le L_\beta\left(\frac{\eta}{\beta}r_i\right)^s
 =\bigl[L_\beta(\eta/\beta)^s\bigr]r_i^s
 \le (\beta r_i)^s,
\]
where the sum is over the branches selected by Alice on this turn.
Thus Alice's choice satisfies~\eqref{eq:budget}.

Suppose $r_i\to0$, and let $x_\infty\in\bigcap_{i\ge0}B_i$
lie outside every ball chosen by Alice during the whole game. For other cases, Alice wins by default. If $x_\infty\notin K$, then it belongs to $S(K,y)$ and Alice wins. So assume that $x_\infty \in K$, we shall show that $x_\infty\in E(f, y)$. If $f^n(x_\infty)\notin B(y,\epsilon)$ for all $n\ge0$, then clearly $x_\infty\in E(f,y)$. Otherwise, suppose that $f^n(x_\infty)\in B(y,\epsilon)$
for some $n\ge0$. Lemma~\ref{lem:captureK} gives a branch
$\phi\in\cG_n$ with $x_\infty\in\phi(B(y,\epsilon))$.
By \eqref{eq:small-image} and
$H\epsilon/t\le\eta$, we have
\[
 x_\infty\in\phi(B(y,\epsilon))\cap J(f)
 \subset \overline B_J(p_\phi,\eta D_\phi).
\]
By \eqref{eq:diam-global}, $D_\phi<r_0$. Since $r_i \to 0$, there is a unique index $i\ge1$ with $r_i\le D_\phi < r_{i-1}$.
Thus $x_\infty\in B_i\cap \overline B_J(p_\phi,\eta D_\phi)$.
Therefore by~\eqref{eq:crossingK}, this ball is one of the balls
chosen by Alice on turn $i$, a contradiction.
\end{proof}

\begin{proof}[Proof of Theorem~\ref{thm:reduction}]
Fix $y\in J(f)$. We prove that $E(f,y)$ is absolute winning
on $J(f)$ under the assumption that $\Ecrit$ is absolute winning. For each $m\ge1$, put
\[
 K_m=\left\{x\in J(f):
 \dist(f^n(x),\CJ)\ge m^{-1}\text{ for every }n\ge0\right\}.
\]
Each $K_m$ is compact and forward invariant and $\Ecrit=\bigcup_{m\ge1}K_m$. By
Definition~\ref{def:away}, $f$ is uniformly expanding on every
nonempty $K_m$.
Proposition~\ref{prop:expanding} shows that $S(K_m,y)$ is absolute winning on $J(f)$ for each $m$. Since
\begin{equation}\label{eq:intersection-identity}
 \Ecrit\cap\bigcap_{m\ge1}S(K_m,y)
 =\Ecrit\cap E(f,y)
\end{equation}
we have that $\Ecrit\cap E(f,y)$
is absolute winning on $J(f)$. Hence $E(f,y)$ is absolute winning as well. Since $y\in J(f)$ was arbitrary,
the theorem follows.
\end{proof}

\section{Backward contraction and absolute winning}
\label{sec:critical}

In this section, let $f\in\sA_{\mathbb C}$, with
$\CJ\ne\varnothing$, be expanding away from critical points
and backward contracting, and put $\ell=\ell_{\max}(f)$.
The main aim of this section is to prove the following proposition,
from which Theorem~\ref{thm:rls} is easily deduced.

\begin{prop}\label{prop:critical-winning}
Let $f\in\sA_{\mathbb C}$, with $\CJ\ne\varnothing$, be expanding
away from critical points and backward contracting. Then
$\Ecrit$ is absolute winning on $J(f)$.
\end{prop}

\begin{proof}[Proof of Theorem~\ref{thm:rls}]
By Proposition~\ref{prop:critical-winning}, $\Ecrit$ is
absolute winning on $J(f)$. Theorem~\ref{thm:reduction}
then implies that $E(f,y)$ is absolute winning on $J(f)$
for every $y\in J(f)$.
\end{proof}

\subsection{Nice couple and pull-backs}

Following~\cite{RLS}, a set $V\subset\CC$
is called a {\it nice set} if the following conditions hold:
\begin{enumerate}
\item $\overline V$ is disjoint from the forward orbits of each point in
$\Crit(f)\setminus J(f)$ and every periodic orbit not in $J(f)$;
\item $\CJ\subset V$;
\item each connected component of $V$ is an open topological
 disk containing exactly one point in $\CJ$;
\item $f^n(\partial V)\cap V=\varnothing$ for each integer $n\ge1$.
\end{enumerate}
For $c\in\CJ$, write $V^c$ for the connected component of $V$
that contains $c$.
A pair $(\hV,V)$ of nice sets is called a {\it nice couple} if
\[
 \overline V\subset\hV
 \ \text{and}\ f^n(\partial V)\cap\hV=\varnothing
 \ \text{for all }n\ge1.
\]
Pull-backs of a nice set are either disjoint or nested.

For open sets $U,V\subset\CC$ with $\overline U\subset V\subsetneq\CC$,
we define $\Mod(V;U)$ as the supremum of the moduli of annuli in
$V$ separating $U$ from $\CC\setminus V$, with the normalization
$\Mod\{1<|z|<a\}=\log a$ for $a>1$.

Given a nice set $V$, a connected component of the set $V\cap\bigcup_{n\ge1}f^{-n}(V)$ is called a {\it return domain}.
These are maximal pull-backs of $V$ contained
in $V$. Given $\Lambda>0$, we say that $V$ is
{\it $\Lambda$-nice} if every return domain $Q\subset V$
satisfies
\[
 \overline Q\subset V
 \ \text{and}\ \Mod(V; Q)\ge\Lambda.
\]

The {\it conical Julia set} $\Jcon(f)$ is the set of points
$x\in J(f)$ for which there exist a radius $\rho_x>0$ and
infinitely many integers $n\ge1$ such that $f^n$ has a
univalent inverse branch on $B(f^n(x),\rho_x)$ taking
$f^n(x)$ to $x$.

We use the following results from~\cite{RLS}.

\begin{prop}\label{prop:rls-input}
Under the hypotheses of Proposition~\ref{prop:critical-winning},
the following hold.
\begin{enumerate}
\item For every $b>0$, $f$ satisfies the Polynomial Shrinking
Condition with exponent $b$: there are
$\rho_0,C_b>0$ such that for each $x\in J(f)$, each $n \ge 1$ and each pull-back $W$ of $B(x,\rho_0)$ under $f^n$, we have $\diam(W)\le C_b n^{-b}$.
\item Given $\Lambda>0$, there is a constant $R>4$ such that
the following property holds.
For every sufficiently small $u>0$, there is a nice
couple $(\hV,V)$ such that $V$ is $\Lambda$-nice and, for each
$c\in\CJ$,
\begin{equation}\label{eq:nice-sandwich}
 \tB(c,Ru/2)\subset\hV^c\subset\tB(c,Ru), \  \tB(c,u)\subset V^c\subset\tB(c,2u).
\end{equation}
\item The conical Julia set satisfies
\begin{equation}\label{eq:nonconical}
 \HD(J(f)\setminus\Jcon(f))=0.
\end{equation}
\end{enumerate}
\end{prop}

Statement~(1) follows from~\cite[Theorem~A]{RLS},
statement~(2) from~\cite[Lemma~3.13]{RLS}, and
statement~(3) from~\cite[Corollary~6.3]{RLS}.
The proof of~\cite[Lemma~3.13]{RLS} also allows the constant
$R$ in~(2) to be chosen larger than any prescribed constant.

\begin{lemm}[{\cite[Lemma~3.3]{RLS}}]\label{lem:modulus-diameter}
Fix $0<\rho_*<\diam(\CC)/2$.
There exists a constant $C_0>0$ such that, for all nonempty open
connected sets $U,V\subset\CC$ with $\overline U\subset V$ and
$\diam(V)<\rho_*$, we have
\[
 \diam(U)\le C_0\exp\bigl(-\Mod(V;U)\bigr)\diam(V).
\]
\end{lemm}

\begin{lemm}\label{lem:small-pullbacks}
For every $h>0$ there is $u_0>0$ such that the following holds. For each $c\in\CJ$, each $n \geq 0$ and each pull-back $W$ of $\tB(c,u_0)$ under $f^n$, we have $\diam(W) < h$. For this fixed $u_0$, let $W$ range over all pull-backs of
$\tB(c,u_0)$ under $f^n$; then
\[
 \sup_{c\in\CJ}\;
 \sup_W
 \diam(W)\longrightarrow0
 \ \text{as }n\to\infty.
\]
\end{lemm}

\begin{proof}
Apply Proposition~\ref{prop:rls-input} (1) with $b=1$, and let
$\rho_0,C_1>0$ be the constants given there.
By \eqref{eq:local-degree}, there exists $u_*>0$ such that, for every
$0<u\le u_*$ and every $c\in\CJ$, we have
$\tB(c,u)\subset B(c,\rho_0)$.
Choose an integer $N\ge1$ such that $C_1/N<h$.
Every pull-back $W$ of $\tB(c,u)$ under $f^n$, with $n\ge N$,
is contained in a pull-back of $B(c,\rho_0)$ under $f^n$ and hence
satisfies $\diam(W)\le C_1/n\le C_1/N<h$.

It remains to treat the finitely many times $0\le j<N$.
For each $c\in\CJ$ and each $0\le j<N$, we claim that
\[
 \sup_W\diam(W)\longrightarrow0
 \ \text{as }r\to0,
\]
where $W$ ranges over all pull-backs of $B(c,r)$ by $f^j$.

Otherwise, there exist $\epsilon_0>0$ and a sequence $r_k\to0$,
and pull-backs $W_k$ of $B(c,r_k)$ by $f^j$ such that
$\diam(W_k)\ge\epsilon_0$ for every $k$.
Passing to a subsequence, the compact connected sets
$\overline{W_k}$ converge in the Hausdorff metric to a compact
connected set $F$ with $\diam(F)\ge\epsilon_0$.
By continuity of $f^j$, we have $f^j(F)=\{c\}$, so
$F\subset f^{-j}(c)$. Since $f^{-j}(c)$ is finite and $F$ is
connected, $F$ must be a singleton, contradicting
$\diam(F)\ge\epsilon_0$.

Since there are only finitely many choices of $c$ and $j<N$,
we can choose a common $r>0$ such that all these components
have diameter less than $h$. Choose $u_0\in(0,u_*]$ sufficiently small
that $\tB(c,u_0)\subset B(c,r)$ for every
$c\in\CJ$. This proves the first statement.
The final statement follows from the fact that
$\diam(W)\le C_1n^{-1}$ for $n\ge1$.
\end{proof}

By Lemma~\ref{lem:small-pullbacks}, taking $u$ sufficiently small in
\eqref{eq:nice-sandwich} ensures that every pull-back of
$\hV^c$, $c\in\CJ$, has diameter less than $\rho_*$.

\subsection{Geometry of pull-backs}

The following lemma uses backward contraction to obtain a
univalent inverse branch at the first entry into a neighbourhood of $\CJ$.

\begin{lemm}\label{lem:first-entry}
Let $R>4$ and $u>0$ be small, and let $(\hV,V)$ be a nice couple
satisfying \eqref{eq:nice-sandwich}.
Let $0<d<u$ be such that $\tB(\tilde c,d)\cap J(f)=\varnothing$
for every $\tilde c\in\Crit(f)\setminus J(f)$.

Assume that, for each $c\in\CJ$, $m\ge0$, and each pull-back
$Q$ of $\tB(c,Ru)$ by $f^m$,
\begin{equation}\label{eq:two-scales}
 \dist(Q,f(\Crit(f)))\le d \ \Longrightarrow \ \diam(Q)<d.
\end{equation}
If $w\in J(f)$, $t\ge0$, and $c\in\CJ$ satisfy
\[
 f^t(w)\in\tB(c,d), \ 
 f^k(w)\notin\tB(\CJ,d) \mbox{ for } 0\le k<t,
\]
then $f^t:P\to\hV^c$ is univalent, where $P$ is
the connected component of $f^{-t}(\hV^c)$ containing $w$.
\end{lemm}

\begin{proof}
For $t=0$ there is nothing to prove. For each $0\le k\le t$,
let $P_k$ be the connected component of
$f^{-(t-k)}(\hV^c)$ containing $f^k(w)$.
Then $P_t=\hV^c$, $P_0=P$, and $f(P_k)=P_{k+1}$ for $0\le k<t$.
Suppose that $c'\in\Crit(f)\cap P_k$ for some $k<t$.
Then $f(c')\in P_{k+1}$. Let $Q$ be the pull-back
of $\tB(c,Ru)$ by $f^{t-k-1}$ containing $P_{k+1}$. Since $Q$
meets $f(\Crit(f))$, \eqref{eq:two-scales} gives $\diam(Q)<d$.
In particular $P_{k+1}\subset B(f(c'),d)$, and hence
\[
 P_k\subset\tB(c',d).
\]
If $c'\in\CJ$, this contradicts $f^k(w)\notin\tB(\CJ,d)$. If $c'\notin J(f)$, the additional assumption on $d$ contradicts
$f^k(w)\in P_k\cap J(f)$. Thus no $P_k$, $k<t$, contains a critical point. This finishes the proof.
\end{proof}

Fix $R>4$ and $u>0$ small, let $(\hV,V)$ be a nice couple
satisfying \eqref{eq:nice-sandwich}.
For each $c\in\CJ$ and $n\ge0$, consider all inverse branches
$\phi$ of $f^n$ on $\hV^c$. Associate to each such branch
\begin{equation}\label{eq:critical-parent}
 W=\phi(V^c), \  p_W=\phi(c), \  D_W=\diam(W).
\end{equation}
Let $\cW$ be the collection of all domains $W = \phi(V^c)$, where
$c\in\CJ$, $n\ge0$, and $\phi$ ranges over all inverse branches
of $f^n$ on $\hV^c$. For $n=0$, the inverse branch is the identity,
so $V^c\in\cW$ for every $c\in\CJ$. The family $\cW$ is countable. Each $W\in\cW$ has a unique pull-back time $n$, with
$f^n(W)=V^c$. Hence $p_W=(f^n|_W)^{-1}(c)$ is well defined.

\begin{lemm}\label{lem:critical-geometry}
There exist $\Lambda>0$, $R>4$, $\hat u>0$, $0<\xi<1$,
$C\ge1$, and an integer $N_b\ge1$ with the following property.
For each $u\in(0,\hat u)$ and each nice couple $(\hV,V)$
satisfying \eqref{eq:nice-sandwich} and such that $V$ is $\Lambda$-nice,
the following estimates hold for every $W=\phi(V^c)\in\cW$
with $c\in\CJ$ and every $0<d<u$:
\begin{equation}\label{eq:critical-inner}
 B(p_W,\xi D_W)\subset W,
\end{equation}
\begin{equation}\label{eq:critical-hole}
 \phi(\tB(c,d))\subset
 B\bigl(p_W,C(d/u)^{1/\ell}D_W\bigr).
\end{equation}
If $W',W\in\cW$ and $W'\subsetneq W$, then
\begin{equation}\label{eq:nested-contraction}
 \diam(W')\le\tfrac12\diam(W).
\end{equation}
Moreover, for each $x\in J(f)$, each $\rho>0$, $0<r\le\rho$ and
$0<\eta\le1$, there are at most $N_b$
domains $W\in\cW$ satisfying
\begin{equation}\label{eq:critical-band}
 \rho\le D_W<2\rho, \ 
 \overline B_J(p_W,\eta D_W)\cap \overline B_J(x,r)\ne\varnothing.
\end{equation}
\end{lemm}

\begin{proof}

Let $a_0,a_1>0$ be the constants in \eqref{eq:local-degree}.
Let $C_0$ be the constant in Lemma~\ref{lem:modulus-diameter}, and choose
$\Lambda>0$ such that $C_0e^{-\Lambda}\le1/2$.
Choose $R$ large enough as in Proposition~\ref{prop:rls-input} (2) so that
\[
 4a_1 2^{1/\ell_c}\le a_0(R/2)^{1/\ell_c}
 \ \text{for every }c\in\CJ.
\]
Then, for each $c\in\CJ$ and all sufficiently small $u$,
\[
 B(c,a_0u^{1/\ell_c})\subset V^c\subset B(c,a_1(2u)^{1/\ell_c}), \ 
 B(c,4a_1(2u)^{1/\ell_c})\subset\hV^c.
\]
Let $r_{\mathrm K}$ be the constant in
Proposition~\ref{prop:spherical-koebe}. Choose $\hat u>0$ sufficiently
small so that, for every $0<u<\hat u$ and every $c\in\CJ$,
\[
 2a_1(2u)^{1/\ell_c}<r_{\mathrm K},
\]
and every pull-back of $\hV^c$ has diameter less than $\rho_*$.

Apply Proposition~\ref{prop:spherical-koebe} with $v=c$ and
$r=2a_1(2u)^{1/\ell_c}$ to the restriction of $\phi$ to
$B(c,2r)\subset\hV^c$. By \eqref{eq:spherical-koebe-balls}, there exists
a constant $K_0\ge1$, independent of $u,c,n,\phi$, such that
\[
 K_0^{-1}|D\phi(c)|u^{1/\ell_c}
 \le D_W\le K_0|D\phi(c)|u^{1/\ell_c}.
\]
Increasing $K_0$ if necessary, we also have
\[
 B\!\left(p_W,K_0^{-1}|D\phi(c)|u^{1/\ell_c}\right)
 \subset\phi\!\left(B(c,a_0u^{1/\ell_c})\right)\subset W.
\]
Taking
$\xi=1/(2K_0^2)$ gives
\[
 B(p_W,\xi D_W)\subset W.
\]
This proves \eqref{eq:critical-inner}. For each $0<d<u$ and
each $z\in\tB(c,d)$, the same estimates give
\[
 d\bigl(\phi(z),\phi(c)\bigr)
 \le K_1|D\phi(c)|d^{1/\ell_c},
\]
where $K_1$ is independent of $u,d,c,n,\phi$.
Since $0<d/u<1$ and $\ell_c\le\ell$,
\[
 (d/u)^{1/\ell_c}\le(d/u)^{1/\ell}.
\]
Thus \eqref{eq:critical-hole} holds with $C=\max\{1,K_0K_1\}$.

We next prove \eqref{eq:nested-contraction}. Let
$W=\phi(V^c)$ and $W'=\phi'(V^{c'})\subsetneq W$, with
pull-back times $n < m$, resp. The domain $f^n(W')$ is a positive-time pull-back of $V$
contained in $V^c$. Therefore it is contained in a return
domain $Q\subset V^c$. Since $V$ is $\Lambda$-nice and
$f^n:W\to V^c$ is univalent, conformal invariance and
monotonicity of modulus give
\[
 \Mod(W;W')=\Mod(V^c;f^n(W'))\ge\Mod(V^c;Q)\ge\Lambda.
\]
By Lemma~\ref{lem:modulus-diameter} we have
\[
 \diam(W')\le C_0\exp\bigl(-\Mod(W;W')\bigr)\diam(W)
 \le C_0e^{-\Lambda}\diam(W)\le\tfrac12\diam(W),
\]
provided $\diam(W)<\rho_*<\diam(\CC)/2$ when $u$ is small.

Finally, we prove the uniform bound on the number of domains
satisfying \eqref{eq:critical-band}.
Distinct domains satisfying \eqref{eq:critical-band} are disjoint:
if one contained the other, \eqref{eq:nested-contraction} would
contradict $\rho\le D_W<2\rho$.
For each $W$ satisfying \eqref{eq:critical-band}, we have
$d(p_W,x)\le r+\eta D_W<3\rho$. By
\eqref{eq:critical-inner}, the spherical balls $B(p_W,\xi\rho)$
are pairwise disjoint. Since $\xi<1$, all these
balls lie in $B(x,4\rho)$. An easy area argument shows that the number of such domains is at most
\[
 \frac{\operatorname{Area}(B(x,4\rho))}
 {\operatorname{Area}(B(x,\xi\rho))}
 \le\left(\frac{4\rho}{\xi\rho}\right)^2
 =\frac{16}{\xi^2}.
\]
Thus we may take $N_b=\lceil16/\xi^2\rceil$, independently of
$u,\rho,r,\eta$.
\end{proof}

\subsection{Proof of Proposition~\ref{prop:critical-winning}}
\begin{proof}
By Proposition~\ref{prop:conversion}, it suffices to prove that
$\Ecrit$ is $0$-potential winning on $J(f)$.
Fix $s>0$ and $\beta\in(0,1)$, and put
\[
 N_\beta=2+\left\lceil\log_2(1/\beta)\right\rceil.
\]
Choose the constants in Lemma~\ref{lem:critical-geometry}.
For each $j\ge1$, choose $\eta_j\in(0,1)$ and
$\tau_j\in(0,1)$ such that
\begin{equation}\label{eq:critical-budget-choice}
 N_bN_\beta(\eta_j/\beta)^s\le2^{-(j+1)}\beta^s,
  \  C\tau_j^{1/\ell}\le\eta_j.
\end{equation}
For each $j \geq 1$, choose $u_j\in(0,\hat u)$ sufficiently small and
set $d_j=\tau_j u_j$. Let $(\hV_j,V_j)$ be the nice couple
given by Proposition~\ref{prop:rls-input} (2) for the chosen
$\Lambda,R$ at scale $u=u_j$. Thus $V_j$ is $\Lambda$-nice and
\eqref{eq:nice-sandwich} holds with $u=u_j$.
Furthermore, the choice of $u_j$ should guarantee the following:

\begin{enumerate}
\item $d_j<\delta_0(R/\tau_j)$, so that
$BC(R/\tau_j)$ gives \eqref{eq:two-scales}
with $\delta=d_j$;
\item every pull-back of every component of $\hV_j$ has diameter
less than $\min\{\rho_*,1/j\}$;
\item $\tB(c,d_j)$ is disjoint from $J(f)$ for each
$c\in\Crit(f)\setminus J(f)$.
\end{enumerate}
Condition~(2) follows
from Lemma~\ref{lem:small-pullbacks}. Let $\cW_j$ be the collection
in \eqref{eq:critical-parent}
for $(\hV_j,V_j)$. By \eqref{eq:critical-hole}, for each
associated inverse branch,
\begin{equation}\label{eq:critical-deletion}
 \phi(\tB(c,d_j))\cap J(f)\subset \overline B_J(p_W,\eta_jD_W).
\end{equation}

Bob first chooses ball $B_0=\overline B_J(b_0,r_0)$. Since $s>0$,
\eqref{eq:nonconical} implies that
$\mathcal H^s(J(f)\setminus\Jcon(f))=0$.
By the definition of Hausdorff measure~\cite{Falconer},
Alice can cover
$(J(f)\setminus\Jcon(f))\cap B_0$ with balls centred in $J(f)$
whose total $s$-power radius is at most
$\tfrac12(\beta r_0)^s$. These balls are her first choice.

At each turn $i\ge1$, after Bob chooses $B_i=\overline B_J(b_i,r_i)$,
Alice chooses the balls $\overline B_J(p_W,\eta_jD_W)$ corresponding
to all $j\ge1$ and all $W\in\cW_j$ satisfying
\begin{equation}\label{eq:critical-crossing}
 r_i\le D_W<r_{i-1}, \ \overline B_J(p_W,\eta_jD_W)\cap B_i\ne\varnothing.
\end{equation}
Since $r_i\ge\beta r_{i-1}$,
\[
 [r_i,r_{i-1})\subset[r_i,r_i/\beta)
 \subset\bigcup_{q=0}^{N_\beta-1}[2^qr_i,2^{q+1}r_i).
\]
For each fixed $j$ and $q$, apply
Lemma~\ref{lem:critical-geometry} with $x=b_i$, $r=r_i$,
$\rho=2^qr_i$, and $\eta=\eta_j$.
It gives at most $N_b$ domains $W\in\cW_j$ satisfying
$D_W\in[2^qr_i,2^{q+1}r_i)$ and
$\overline B_J(p_W,\eta_jD_W)\cap B_i\ne\varnothing$.
Thus Alice chooses at most $N_bN_\beta$ balls for each $j$.
Therefore the total sum of the $s$-th powers of the radii of Alice's choice is bounded by
\[
 \sum_{j\ge1}N_bN_\beta(\eta_j/\beta)^sr_i^s
 \le\tfrac12(\beta r_i)^s.
\]
Thus Alice's choice satisfies~\eqref{eq:budget}.

Suppose $r_i\to0$, and let $x_\infty\in\bigcap_{i\ge0}B_i$
lie outside every ball chosen by Alice during the whole game.
In all other cases, Alice wins by default.
Since the balls chosen on the first turn cover
$(J(f)\setminus\Jcon(f))\cap B_0$, we have
$x_\infty\in\Jcon(f)$.
If $x_\infty\in\Ecrit$, then Alice wins.
Otherwise, suppose that $x_\infty\notin\Ecrit$. Note that $f^q(y)\notin\Crit(f)$ for every $y\in\Jcon(f)$
and every $q\ge0$.
Since $x_\infty\notin\Ecrit$, we have
\[
 \inf_{n\ge N}\dist(f^n(x_\infty),\CJ)=0
 \ \text{for every }N\ge0.
\]
Consequently, the orbit of $x_\infty$ visits every neighbourhood
of $\CJ$ infinitely many times.

Choose $\rho>0$ and a subsequence $\{n_k\}$ with $n_k\to\infty$
such that, for each $k$, there is a univalent inverse branch
\[
 \psi_{n_k}:B(f^{n_k}(x_\infty),\rho)\longrightarrow\CC,
  \  \psi_{n_k}(f^{n_k}(x_\infty))=x_\infty.
\]
Fix $j$ with $1/j<\rho$. For each $k$, let
\[
 t_k=\min\{m\ge0:f^{n_k+m}(x_\infty)\in\tB(\CJ,d_j)\},
\]
and choose $c_k\in\CJ$ with
$f^{n_k+t_k}(x_\infty)\in\tB(c_k,d_j)$.
We apply Lemma~\ref{lem:first-entry} with $(u,d)=(u_j,d_j)$,
$(\hV,V)=(\hV_j,V_j)$, $w=f^{n_k}(x_\infty)$, $t=t_k$, and $c=c_k$.
Here $0<d_j<u_j$, the nice couple satisfies \eqref{eq:nice-sandwich},
and conditions~(1) and~(3) above give the two additional hypotheses.
The orbit conditions hold by the definition of $t_k$ and the choice of $c_k$.
Thus the lemma gives a univalent inverse branch
\[
 \theta_k:\hV_j^{c_k}\longrightarrow P_k,
  \  \theta_k(f^{n_k+t_k}(x_\infty))=f^{n_k}(x_\infty),
\]
where $P_k$ is the pull-back of $\hV_j^{c_k}$ by $f^{t_k}$
containing $f^{n_k}(x_\infty)$.
By the choice of $u_j$, $\diam(P_k)<1/j<\rho$. Since $P_k$ contains
$f^{n_k}(x_\infty)$, it lies in $B(f^{n_k}(x_\infty),\rho)$.
Thus
\[
 \phi_k=\psi_{n_k}\circ\theta_k
\]
is a univalent inverse of $f^{n_k+t_k}$ on $\hV_j^{c_k}$.
Put $W_{k,j}=\phi_k(V_j^{c_k})$. Then $W_{k,j}\in\cW_j$ is a pull-back
of $V_j^{c_k}$ under $f^{n_k+t_k}$, and
$x_\infty\in\phi_k(\tB(c_k,d_j))$.

The pull-back times satisfy $n_k+t_k\ge n_k\to\infty$.
Since $j$ is fixed, Lemma~\ref{lem:small-pullbacks} gives
\[
 D_{W_{k,j}}=\diam(W_{k,j})\longrightarrow0
 \ \text{as }k\to\infty.
\]
Fix $k$ sufficiently large that $D_{W_{k,j}}<r_0$.
Since $r_i\to0$, there is a unique index $i\ge1$ with
$r_i\le D_{W_{k,j}}<r_{i-1}$.
By~\eqref{eq:critical-deletion}, we have
\[
 x_\infty\in\phi_k(\tB(c_k,d_j))\cap J(f)
 \subset \overline B_J(p_{W_{k,j}},\eta_jD_{W_{k,j}}).
\]
Thus $x_\infty\in B_i\cap \overline B_J(p_{W_{k,j}},\eta_jD_{W_{k,j}})$.
By~\eqref{eq:critical-crossing}, this ball is one of
the balls chosen by Alice on turn $i$, contradicting our
assumption that $x_\infty$ lies outside every ball chosen by Alice.
This proves that $x_\infty\in\Ecrit$.
Since $s$ and $\beta$ were arbitrary,
Proposition~\ref{prop:conversion} proves that $\Ecrit$ is absolute
winning on $J(f)$.
\end{proof}

\section{Interval maps}\label{sec:interval}

In this section, we prove a similar result for smooth interval maps. Throughout, we consider $C^1$ maps $f:I\to I$, where $I=[0,1]$.
Write
\[
 \Crit(f)=\{c\in I:Df(c)=0\}
\]
for the set of critical points of $f$.
We call $f$ a {\it multimodal map} if $\Crit(f)$
is finite. Let $\sA_{\mathbb R}$ denote the class of multimodal maps
satisfying the following conditions:
\begin{enumerate}
\item $f(\{0,1\})\subset\{0,1\}$ and $\Crit(f)\subset(0,1)$.
\item Every periodic orbit is hyperbolic repelling.
\item The map $f$ is $C^3$ on $I\setminus\Crit(f)$, and every
critical point $c$ is {\it nonflat}: there are $C^3$ local
diffeomorphisms $\phi_c,\psi_c$ and a real number
$\ell_c>1$ such that
\begin{equation}\label{eq:nonflat}
 \phi_c(c)=\psi_c(f(c))=0, \ 
 |\psi_c\circ f(x)|=|\phi_c(x)|^{\ell_c}
\end{equation}
on a neighborhood of $c$. The number $\ell_c$ is called the order
of $c$. We also put 
\[
 \ell=\ell_{\max}(f)=\max_{c\in\Crit(f)}\ell_c.
\]

\item $f$ is {\it topologically exact} on $I$: for every
nonempty relatively open set $U\subset I$, there exists an
integer $n\ge1$ such that $f^n(U)=I$.
\end{enumerate}

The real Julia set $J(f)$ for smooth interval maps is defined as the complement of the basins of periodic attractors, see \cite{MS}. Hence for each $f \in \sA_{\mathbb R}$, we have $J(f) = I$.

All distances, diameters, and derivative norms are taken with respect
to the Euclidean metric. Write $B(x,r)=(x-r,x+r)$ for open balls in
$\mathbb R$ and $\overline B_I(x,r)=I\cap[x-r,x+r]$ for game balls on $I$.
The definitions of expansion away from critical points and backward
contraction given in Section~1 also apply to interval maps, with the
spherical metric replaced by the Euclidean metric. Here, for $c\in\Crit(f)$
and $\delta>0$, $\tB(c,\delta)$ denotes the connected component of
$f^{-1}(B(f(c),\delta)\cap I)$ containing $c$. We also put
\[
 \tB(\Crit(f),\delta)=\bigcup_{c\in\Crit(f)}\tB(c,\delta).
\]
Note that by Ma\~n\'e's theorem~\cite{Man85}, each
$f\in\sA_{\mathbb R}$ is expanding away from critical points.

We spell out the definition. Let $r>1$. We say that $f\in\sA_{\mathbb R}$ is
{\it backward contracting with exponent $r$}, abbreviated $BC(r)$,
if there is $\delta_0=\delta_0(r)>0$ such that, for every
$\delta\in(0,\delta_0)$, every $c\in\Crit(f)$, every $n\ge0$, and
every pull-back $W$ of $\tB(c,r\delta)$ by $f^n$,
\begin{equation}\label{eq:real-bc}
 \dist(W,f(\Crit(f)))\le\delta \ \Longrightarrow \  |W|<\delta.
\end{equation}
The map $f$ is called {\it backward contracting}
if it satisfies $BC(r)$ for every $r>1$. The local normal form gives constants $a_0,a_1>0$ such that
\begin{equation}\label{eq:real-local}
 B(c,a_0u^{1/\ell_c})\subset\tB(c,u)
 \subset B(c,a_1u^{1/\ell_c})
\end{equation}
for every $c\in \Crit(f)$ and all sufficiently small $u>0$.

The following remarkable proposition was proved in~\cite[Theorem~1]{BRLSvS08}.
For the relation between backward contraction, large derivatives,
and summability conditions, see~\cite{LS10,Li14}.

\begin{prop}
Let $f\in\sA_{\mathbb R}$ satisfy
\[
 |Df^n(f(c))|\longrightarrow\infty
 \ \text{as }n\to\infty
 \ \text{for each }c\in\Crit(f).
\]
Then $f$ is backward contracting.
\end{prop}

For $y\in I$, put
\[
 E_I(f,y)=\{x\in I:\inf_{n\ge0}|f^n(x)-y|>0\}.
\]
We call this the {\it exceptional set} associated to $y$. Our main result for interval maps is the following.

\begin{theoalph}\label{thm:real}
Let $f\in\sA_{\mathbb R}$ be backward contracting.
Then $E_I(f,y)$ is absolute winning on $I$ for every $y\in I$.
\end{theoalph}

For $Y\subset I$, define
\[
 E_I(f,Y)=\{x\in I:\overline{\mathcal O_f^+(x)}\cap Y=\varnothing\}
 =\bigcap_{y\in Y}E_I(f,y).
\]
Theorem~\ref{thm:real} and the countable intersection property imply
that $E_I(f,Y)$ is absolute winning on $I$ for every countable set
$Y\subset I$.

The game definitions and Propositions~\ref{prop:potential-properties}
and~\ref{prop:conversion} apply to every compact interval with its
Euclidean metric.

\subsection{Real pull-back estimates}

Mimicking the proof of Proposition~\ref{prop:expanding}, we obtain
the following result.

\begin{prop}\label{prop:real-expanding}
Let $f\in\sA_{\mathbb R}$, and let $K\subset I$ be compact and
forward invariant. Suppose that $f$ is uniformly expanding on $K$.
Then, for each $y\in I$,
\begin{equation}\label{eq:real-conditional-set}
 S(K,y)=(I\setminus K)\cup E_I(f,y)
\end{equation}
is absolute winning on $I$.
\end{prop}

\begin{proof}
As in Proposition~\ref{prop:expanding}, the case $K=\varnothing$
or $y\notin K$ is trivial, so we assume that $y\in K$.
There are two points requiring modification. First, for inverse
branches within $I$, we use the real Koebe principle
\cite[Proposition~1]{BRLSvS08} in place of Proposition~\ref{prop:spherical-koebe} to obtain the real analogues of \eqref{eq:large-image}--\eqref{eq:branch-lip}.

Second, for $x\in K$, the interval $B(f^n(x),4\delta)$ may extend
outside $I$ when $f^n(x)$ is close to an endpoint.
Since $Df(0),Df(1)\ne0$, we extend $f$ as a $C^3$ map locally
beyond the endpoints before applying the argument of
Lemma~\ref{lem:extension}, taking $O$ to be an open neighbourhood
of $K$ in $\mathbb R$ and choosing $\delta>0$ sufficiently small.
Choose $O$ so that $L:=\sup_O|D\log|Df||<\infty$, and let
$A,\lambda$ be the corresponding expansion constants.
For each extended inverse branch $\phi$ of $f^n$ on $B(y,2\delta)$
whose intermediate images lie in $O$, the chain rule and
\eqref{eq:segment-expansion} give, for $u_1,u_2\in B(y,2\delta)$,
\[
 \left|\log\frac{|D\phi(u_1)|}{|D\phi(u_2)|}\right|
 \le L\sum_{k=0}^{n-1}|f^k(\phi(u_1))-f^k(\phi(u_2))|
 \le\frac{L}{A(\lambda-1)}|u_1-u_2|.
\]
Thus the distortion is bounded independently of $n$ and $\phi$,
and the real analogues of
\eqref{eq:large-image}--\eqref{eq:branch-lip} follow.
Alice chooses only balls whose centres $\phi(y)$ lie in $I$.
The branch supplied by the real analogue of
Lemma~\ref{lem:captureK} satisfies this condition, since
$f^n(\partial I)\subset\partial I$.
The proofs of Lemma~\ref{lem:countK} and
Proposition~\ref{prop:expanding} then remain valid.

 \end{proof}

For the remainder of this subsection, assume that
$f\in\sA_{\mathbb R}$ is backward contracting.
The real versions of Proposition~\ref{prop:rls-input} follow from
\cite[Theorem~A, Lemma~3.13 and Corollary~6.3]{RLS}; here topological
exactness on $I=J(f)$ implies essential topological exactness.
In particular, given $\Lambda>0$, one can choose $R>4$ arbitrarily
large such that, for every sufficiently small $u>0$, there is a
nice couple $(\hV,V)$ with $V$ $\Lambda$-nice and, for every
$c\in\Crit(f)$,
\begin{equation}\label{eq:real-sandwich}
 \tB(c,Ru/2)\subset\hV^c\subset\tB(c,Ru), \ 
 \tB(c,u)\subset V^c\subset\tB(c,2u).
\end{equation}

Here $\Lambda$-niceness is understood as follows: $V$ is
$\Lambda$-nice if $(1+2\Lambda)Q\subset V^c$ for each return domain
$Q\subset V^c$, where $(1+2\Lambda)Q$ denotes the interval with the
same centre as $Q$ and length $(1+2\Lambda)|Q|$. This definition is
equivalent to the modulus-based definition of $\Lambda$-niceness used in
Section~\ref{sec:critical}; see also~\cite{RLS}. The parameter $\Lambda$
may however differ.

Moreover,
\begin{equation}\label{eq:real-conical}
 \HD(I\setminus \Jcon(f))=0,
\end{equation}
where balls in the definition of $\Jcon(f)$
are understood as $B(z,\rho)\cap I$. The proof of
Lemma~\ref{lem:small-pullbacks} applies to the real case
as well, with diameter replaced by length.

Fix $R>4$ and $u>0$ small, and let $(\hV,V)$ be a nice couple
satisfying \eqref{eq:real-sandwich}.
For each $c\in\Crit(f)$ and $n\ge0$, consider all diffeomorphic
inverse branches $\phi$ of $f^n$ on $\hV^c$. Put
\[
 W=\phi(V^c), \ p_W=\phi(c), \ D_W=|W|.
\]
Let $\cW$ be the collection of all these domains, for all
$c\in\Crit(f)$, $n\ge0$, and such branches $\phi$.
For $n=0$, the inverse branch is the identity, so $V^c\in\cW$
for every $c\in\Crit(f)$. The family $\cW$ is countable.
As in Section~\ref{sec:critical}, each $W\in\cW$ has a unique
pull-back time $n$, with $f^n(W)=V^c$. Hence
$p_W=(f^n|_W)^{-1}(c)$ is well defined.

\begin{lemm}\label{lem:real-geometry}
There exist $\Lambda>0$, $R>4$, $\hat u>0$, $0<\xi<1$,
$C\ge1$, and an integer $N_{\mathrm r}\ge1$ with the following property.
For each $u\in(0,\hat u)$ and each nice couple $(\hV,V)$
satisfying \eqref{eq:real-sandwich} and such that $V$ is $\Lambda$-nice,
the following estimates hold for every $W=\phi(V^c)\in\cW$
with $c\in\Crit(f)$ and every $0<d<u$:
\begin{equation}\label{eq:real-inner}
 B(p_W,\xi D_W)\subset W,
\end{equation}
\begin{equation}\label{eq:real-hole}
 \phi(\tB(c,d))\subset B\bigl(p_W,C(d/u)^{1/\ell}D_W\bigr).
\end{equation}
If $W',W\in\cW$ and $W'\subsetneq W$, then
\begin{equation}\label{eq:real-nested}
 D_{W'}\le\tfrac12 D_W.
\end{equation}
Moreover, for each $x\in I$, each $\rho>0$, $0<r\le\rho$,
and $0<\eta\le1$, there are at most $N_{\mathrm r}$ domains $W\in\cW$
satisfying
\begin{equation}\label{eq:real-band}
 \rho\le D_W<2\rho, \ 
 \overline B_I(p_W,\eta D_W)\cap\overline B_I(x,r)\ne\varnothing.
\end{equation}
\end{lemm}

\begin{proof}
We follow the proof of Lemma~\ref{lem:critical-geometry}, with two
changes. First, let $a_0,a_1$ be the constants in
\eqref{eq:real-local}, put $b_c=a_1(2u)^{1/\ell_c}$, and choose $R$
so large that
\[
 4a_1 2^{1/\ell_c}\le a_0(R/2)^{1/\ell_c}
 \ \text{for every }c\in\Crit(f).
\]
For every sufficiently small $u>0$, we then have
\[
 B(c,a_0u^{1/\ell_c})\subset V^c\subset B(c,b_c), \ 
 B(c,4b_c)\subset\hV^c.
\]
By the real Koebe
principle~\cite[Proposition~1]{BRLSvS08}, applied to $f^n$ on
$\phi(B(c,4b_c))$ with $B(c,2b_c)$ $1/2$-well inside $B(c,4b_c)$,
we obtain
\[
 K_0^{-1}\le\frac{|D\phi(z)|}{|D\phi(w)|}\le K_0
 \ \text{for }z,w\in B(c,2b_c), \ 
 K_0=\frac{1}{0.9}\left(\frac{1+1/2}{1/2}\right)^2=10.
\]
In particular, $K_0$ is
independent of $\Lambda,R,u,c,n$ and $\phi$. Since $V^c\subset B(c,b_c)$,
for every interval $L\subset V^c$ we obtain
\begin{equation}\label{eq:real-length}
 K_0^{-1}\frac{|L|}{|V^c|}
 \le\frac{|\phi(L)|}{|\phi(V^c)|}
 \le K_0\frac{|L|}{|V^c|}.
\end{equation}
Since $2a_0u^{1/\ell_c}\le |V^c|\le4a_1u^{1/\ell_c}$,
\eqref{eq:real-length} gives \eqref{eq:real-inner} and
\eqref{eq:real-hole} with
\[
 \xi=\min\{1/2,a_0/(4K_0a_1)\}, \ C=\max\{1,K_0a_1/a_0\}.
\]

Second, we use \eqref{eq:real-length} to compare nested intervals,
in place of conformal invariance of modulus. Let
$W=\phi(V^c)$ and $W'=\phi'(V^{c'})\subsetneq W$, with
pull-back times $n<m$, respectively. The same nesting argument
shows that $L=f^n(W')$ is contained in a return domain $Q\subset V^c$.
By the real $\Lambda$-niceness condition, we have
\[
 \frac{|L|}{|V^c|}\le\frac{|Q|}{|V^c|}\le\frac{1}{1+2\Lambda}.
\]
Hence by \eqref{eq:real-length},
\[
 \frac{|W'|}{|W|}\le\frac{K_0}{1+2\Lambda}\le\frac12
\]
provided $\Lambda\ge K_0$, proving \eqref{eq:real-nested}.
Finally, as in Lemma~\ref{lem:critical-geometry}, the intervals
$B(p_W,\xi\rho)$ for domains satisfying \eqref{eq:real-band}
are pairwise disjoint and contained in $B(x,4\rho)$.
Length comparison gives at most $4/\xi$ such domains, so we may
take $N_{\mathrm r}=\lceil4/\xi\rceil$, independently of $u,\rho,r,\eta$.
\end{proof}

\begin{lemm}\label{lem:real-entry}
Let $R>4$ and $u>0$ be small, and let $(\hV,V)$ be a nice couple
satisfying \eqref{eq:real-sandwich}. Let $0<d<u$.
Assume that, for each $c\in\Crit(f)$, $m\ge0$, and each pull-back
$Q$ of $\tB(c,Ru)$ by $f^m$,
\begin{equation}\label{eq:real-two-scales}
 \dist(Q,f(\Crit(f)))\le d \ \Longrightarrow \  |Q|<d.
\end{equation}
If $w\in I$, $t\ge0$, and $c\in\Crit(f)$ satisfy
\[
 f^t(w)\in\tB(c,d), \ 
 f^k(w)\notin\tB(\Crit(f),d) \ \text{for }0\le k<t,
\]
then $f^t:P\to\hV^c$ is a diffeomorphism, where $P$ is
the connected component of $f^{-t}(\hV^c)$ containing $w$.
\end{lemm}
\begin{proof}
The proof of Lemma~\ref{lem:first-entry} applies without change:
no pull-back interval meets $\partial I$, since
$f(\partial I)\subset\partial I$ and $\hV^c$ lies in the interior of $I$.
\end{proof}

\subsection{Proof of the interval theorem}

\begin{proof}[Proof of Theorem~\ref{thm:real}]

We first prove that $\Ecrit$ is absolute winning on $I$, where
\[
 \Ecrit=\{x\in I:\inf_{n\ge0}\dist(f^n(x),\Crit(f))>0\}
 =E_I(f,\Crit(f)).
\]
With \eqref{eq:real-conical}, Lemma~\ref{lem:real-geometry}, and
Lemma~\ref{lem:real-entry} in place of \eqref{eq:nonconical},
Lemma~\ref{lem:critical-geometry}, and Lemma~\ref{lem:first-entry},
respectively, the proof of Proposition~\ref{prop:critical-winning}
remains valid for the game on $I$.
Consequently, $\Ecrit$ is absolute winning on $I$.

Finally, for each integer $m\ge1$, put
\[
 K_m=\{x\in I:\dist(f^n(x),\Crit(f))\ge m^{-1}
       \text{ for every }n\ge0\}.
\]
Then $\Ecrit=\bigcup_{m\ge1}K_m$, and each $K_m$ is compact and
forward invariant. Expansion away from critical points implies
that $f$ is uniformly expanding on every nonempty $K_m$.
Proposition~\ref{prop:real-expanding}
and the same argument as in the proof of
Theorem~\ref{thm:reduction} then imply that
$E_I(f,y)$ is absolute winning on $I$ for every $y\in I$.
\end{proof}

\bigskip
\needspace{4\baselineskip}
{\small\raggedright
\noindent School of Mathematics and Statistics, Zhengzhou University,\\
Zhengzhou 450001, CHINA\par
\noindent E-mail: \href{mailto:jihymath@zzu.edu.cn}{jihymath@zzu.edu.cn}\par
}
\end{document}